\documentclass[12pt,reqno]{amsart}

\usepackage{amssymb}

\usepackage{tikz-cd}

\def\Z{{\mathbb{Z}}}

\newtheorem{theorem}{Theorem}[section]
\newtheorem{lemma}[theorem]{Lemma}

\theoremstyle{definition} 
\newtheorem{remark}[theorem]{Remark}
\newtheorem{assumption}[theorem]{Assumptions}

\numberwithin{equation}{section}

\begin{document}

\baselineskip=15pt

\title[Stability of parabolic Picard sheaf]{Stability of the parabolic Picard sheaf}

\author[ C. Arusha]{ C. Arusha}
\address{Department of Mathematics, Indian Institute of Technology Bombay, Powai, Mumbai, Maharashtra 400076, India}
\email{arushnmaths@hotmail.com, arusha@math.iitb.ac.in}

\author[I. Biswas]{Indranil Biswas}

\address{Department of Mathematics, Shiv Nadar University, NH91, Tehsil
Dadri, Greater Noida, Uttar Pradesh 201314, India}

\email{indranil.biswas@snu.edu.in, indranil29@gmail.com}

	

\date{}
	
\begin{abstract}
Let $X$ be a smooth irreducible complex projective curve of genus $g\,\geq\, 2$, and let
$D\,=\,x_1+\dots+x_r$ be a reduced effective divisor on $X$. 
Denote by $U_{\alpha}(L)$ the moduli space of stable parabolic vector bundles on $X$ of rank $n$,
determinant $L$ of degree $d$ with flag type $\{\{k^i_j\}_{j=1}^{m_i}\}_{i=1}^r$. Assume that
the greatest common divisor of the collection of integers $\{\text{degree}(L),\,
\{\{k^i_j\}_{j=1}^{m_i}\}_{i=1}^r\}$ is $1$; this condition ensures that there is a Poincar\'e
parabolic vector bundle on $X\times U_{\alpha}(L)$. The direct image, to $U_{\alpha}(L)$, of the
vector bundle underlying the Poincar\'e parabolic vector bundle
is called the parabolic Picard sheaf. We prove that the parabolic Picard sheaf is stable.
\end{abstract}
\maketitle
	
\section{Introduction}

Picard sheaves were introduced in the projectivised form in the most general setting by A. Mattuck \cite{AM}. 
For a smooth irreducible projective curve $X$ of genus $g\,\geq\, 2$, they were studied in \cite{RLES} and 
\cite{GK} for the Jacobian of $X$. The study of stability of Picard sheaves on the moduli spaces of vector 
bundles on $X$ has been carried out in \cite{BBGN} and \cite{IBTG}. The work \cite{BBN} generalizes the 
definition of a Picard sheaf to include all twists of it by a vector bundle and also studies their stability
property.

Although the 
existence of Picard sheaves require the degree and rank to be relatively prime, there is a projective Picard 
bundle in the non co-prime case, and it is stable \cite{BBN2}. The study of Picard and projective Picard bundles 
on the moduli space of vector bundles on an irreducible nodal curve are carried out in \cite{UBAJP}, 
\cite{UNB} and \cite{AUS}. For a complete survey on the existing literature we refer the reader to \cite{BBN} 
and the references therein.
 
In this short note, we fix a reduced effective divisor $D\,=\,x_1+\dots+x_r$ on $X$ and prove that
the Picard sheaf on the moduli space of parabolic bundles on $X$ with a fixed parabolic datum is stable;
see Theorem \ref{main}.
	
\section{Parabolic Picard sheaf}
	
	Let X be a smooth irreducible complex projective curve of genus $g$, with $g\,\geq\, 2$. Fix an integer $n\,\geq\, 2$ and
a line bundle $L$ over $X$ of degree $d$. Also, fix $r$ distinct points $D\,=\, \{x_1,\,\cdots,\,x_r\}\, \subset\, X$, and denote
also by $D$ the divisor $x_1+\dots+x_r$. Let $E$ be a holomorphic vector bundle on $X$ of rank $n$. 
	
	A quasi-parabolic structure over $E$ over $D$ is a strictly decreasing filtration of linear subspaces 
\begin{equation}\label{c1}
E_{x_i}\,=\,F^i_1\,\supset\, F^i_2\,\supset\,\cdots\, \supset\, F^i_{m_i}\,\supset\, F^i_{m_{i+1}}\,=\,0
\end{equation}
for every $x_i\,\in\, D$. Let
\begin{equation}\label{c2}
k^i_j\ :=\ \dim F^i_j-\dim F^i_{j+1}.
\end{equation}
The integer $m_i$ in \eqref{c1} and the sequence $(k^i_1,\,\cdots,\,k^i_{m_i})$ in \eqref{c2}
are respectively called the length and the type of the flag at $x_i$. A parabolic structure on $E$ over
the divisor $D$ is a quasi-parabolic structure as above together with a sequence of real numbers
\begin{equation}\label{c3}
0\,\leq\,\alpha^i_1\,<\, \cdots\,<\,\alpha^i_{m_i}\, <\, 1,
\end{equation}
$1\leq\, i\, \leq\, r$, called the weight system of the flag at $x_i$.
The parabolic degree of $E$ is defined as
$${\rm par }\deg E\,:=\,\deg E+\sum_{x_i\in D}\sum_{j=1}^{m_i}k^i_j\alpha^i_j\,\in\, {\mathbb R}$$
(see \eqref{c2}, \eqref{c3}).

\noindent With the notation same as above, any subbundle $V\,\subset\, E$ has an induced parabolic structure with the quasiparabolic filtration at $x_i$ given by
	$$V_{x_i}\,=\,F^i_1\cap V_{x_i}\,\supset\, F^i_2\cap V_{x_i}\,\supset\,\cdots\,\supset\, F^i_{\ell_i}\cap V_{x_i}\,\supset\, 0$$
where $\ell_i:\,=\, \max\{j\,\in\,\{1,\,\cdots,\, m_i\}\,\big\vert\,\, F^i_j\cap V_{x_i}\,\neq\, 0\}$ and the induced parabolic
weight of $F^i_j$ is $\max\{\alpha^i_{\ell}\, \big\vert\,\,F^i_j\cap V_{x_i}\,=\,F^i_{\ell}\cap V_{x_i}\}$.

\noindent A parabolic vector bundle $E$ with parabolic structure over $D$ is called stable (respectively, semistable) if for every proper
non-trivial subbundle $F\,\subsetneq\, E$ endowed with the induced parabolic structure, the following inequality holds:
	$$\frac{{\rm par }\deg F}{{\rm rk\ } F}\,<\,\frac{{\rm par }\deg E}{{\rm rk\ }E}\ \ \,
\text{(respectively, }\ \frac{{\rm par }\deg F}{{\rm rk\ } F}\,\leq\,\frac{{\rm par }\deg E}{{\rm rk\ }E}\text{).}$$

\noindent Denote by
\begin{equation}\label{cm}
U_X(d,n,\{\alpha^i_j\}, \{k^i_j\})
\end{equation}
the coarse moduli space of semistable parabolic 
vector bundles of rank $n$, degree $d$ with flag type $\{k^i_j\}$ and parabolic weights 
$\{\alpha^i_j\}$ at $x_i$, $1\,\leq\, i\,\leq\, r$. Then $U_X(d,n,\{\alpha^i_j\}, \{k^i_j\})$ is 
a normal projective variety, and the open subvariety
$$
U_X^s(d,n,\{\alpha^i_j\}, \{k^i_j\})\ \subset\ U_X(d,n,\{\alpha^i_j\}, \{k^i_j\})
$$
(see \eqref{cm}) 
parametrizing the stable parabolic bundles is smooth \cite{VMCS}. If the elements of the set 
$\{d,\,k^i_j\,\mid\,\, 1\,\leq\, i\,\leq\, r,\ 1\,\leq\, j \,\leq\, m_i\}$ have greatest common 
divisor equal to one, then
\begin{equation}\label{cm2}
U_{\alpha}^s\ :=\ U_X^s(d,n,\{\alpha^i_j\}, \{k^i_j\})
\end{equation}
is a fine moduli space. In other words, there exists a family of parabolic vector bundles 
\begin{equation}\label{cm3}
\mathcal{U}_{*}^{\alpha}\ :=\ (\mathcal{U},\phi,\alpha)
\end{equation}
parametrized by $U_{\alpha}^s$ such 
that for each $e\,=\,[E_*]\,\in\, U_{\alpha}^s$, we have $\mathcal{U}_{*,e}^{\alpha}$ to be a 
stable parabolic bundle isomorphic to $E_*$ \cite[Proposition 3.2]{HBKY}. In such a situation, 
the notions of stability and semistability coincide and moreover the moduli space 
\begin{equation}\label{c4}
U_{\alpha}\,:=\, U_X(d,n,\{\alpha^i_j\}, \{k^i_j\}) \,=\,  U_X^s(d,n,\{\alpha^i_j\}, \{k^i_j\})\,=:\,
U_{\alpha}^s
\end{equation}
is a smooth irreducible projective variety.

\begin{assumption}\label{ams1}\mbox{}
\begin{enumerate}
\item Henceforth we would assume that the collection of integers
$\{d,\,k^i_j\,\mid\,\, 1\,\leq\, i\,\leq\, r,\ 1\,\leq\, j \,\leq\, m_i\}$ have greatest common
divisor equal to one.

\item We assume that the parabolic structure is non-trivial. This means that for some $1\, \leq\, i
\, \leq\, r$, we have $m_i\, \geq\, 2$.
\end{enumerate}
\end{assumption}

\begin{remark}{\textit{Families of parabolic vector bundles:}}
For a scheme $T$, let
\begin{equation}\label{z1}
\pi_1\,:\,X\times T\,\longrightarrow\, X\ \, \text{ and }\, \
\pi_2\,:\,X\times T\,\longrightarrow\, T
\end{equation}
denote the natural projections. For every
$1\,\leq\, i\,\leq\, r$, let $\mathcal F_{k_i}$ denote the variety of flags of type $k_i\,=\,(k_1^i,\,\cdots,\,k_{m_i}^i)$, where
$\sum_{j=1}^{m_i}k^i_j\,=\,n$. Now, for a rank $n$ vector bundle $E\,\longrightarrow\, T$, denote by
$$\mathcal F_{k_i}(E)\,\longrightarrow\, T$$ the bundle of flags of type $k_i$.
Fix multiplicities $k_i$ for each $x_i\,\in\, D$. A family of quasi-parabolic bundles (of type $k\,=\,(k_1,\, \cdots,\, k_r)$)
parametrized by a scheme $T$ is
a pair $(V,\,\phi)$ where $V$ is a vector bundle $V\,\longrightarrow\, X\times T$ and $\phi$ is a collection of sections
$$\{\phi_{x_i}\,:\,T\,\to\,\mathcal F_{k_i}(V|_{x_i\times T})\}_{1\leq i\leq r}.$$ A family of parabolic bundles is then a triple
$$V_*\,=\,(V,\,\phi,\,\alpha),$$ where $\alpha$ associates weights $\{\alpha^i_j\}$ to the flag of subbundles over $x_i\times S$ for
each $x_i\,\in\, D$. Let $(V,\,\phi)$ and $(V',\,\phi')$ be families parametrized by $T$. We say $(V,\,\phi)$ is equivalent to
$(V',\,\phi')$ if there exists a line bundle $N$ over $T$ and an isomorphism
$$V\,\,\cong\,\, V'\otimes\pi_2^*N,$$
where $\pi_2$ is the projection in \eqref{z1}, that sends $\phi$ to $\phi'$.
\end{remark}
	
\begin{remark}\label{defn}
There is a determinant morphism $$\det\,:\, U_{\alpha}\,\longrightarrow\, J^d(X)$$
(see \eqref{c4}),
where $J^d(X)$ denotes the component of the Picard group of $X$ consisting of line bundles of degree $d$; this map
sends a parabolic vector bundle to the top exterior product of the underlying vector bundle.
For $L\,\in\, J^d(X)$, let
\begin{equation}\label{d-1}
U_{\alpha}(L) \ :=\ {\det}^{-1}(L)
\end{equation}
denote the fiber of the map $\det$ over $L$, and let
\begin{equation}\label{d-2}
{\mathcal U}_L \ :=\ \mathcal U_*^{\alpha}\big\vert_{U_{\alpha}(L)}
\end{equation}
be the restriction of ${\mathcal U}_*^{\alpha}$ (see \eqref{cm3}) to $X\times U_{\alpha}(L)$ (see \eqref{d-1}). 

As before, let $\pi_2\, :\, X\times U_{\alpha}(L)\, \longrightarrow\, U_{\alpha}(L)$ be the natural projection.
The direct image sheaf 
\begin{equation}\label{k1}
{\mathcal W}_L:= {\pi_2}_*{\mathcal U}_L
\end{equation}
on $U_{\alpha}(L)$ is called the \textit{parabolic Picard sheaf}. 
\end{remark}
	
\section{Stability}
	
\subsection{Determinant bundle on $U_{\alpha}(L)$}
	
We fix a closed point $x\,\in\, X$ and rational parabolic weights $0\,\leq\,\alpha_1
\,<\,\alpha_2\,<\,\cdots\,<\,\alpha_n\,<\,1$. Let $\mathcal V_*^{\alpha}\,=\,(\mathcal V,\, \phi,\, \alpha)$ be a family of rank $n$
stable parabolic bundles on $X$ parametrized by a variety $T$ with $\phi$ determining the full flag of subbundles
\begin{equation}\label{c-1}
\mathcal V|_{x\times T}\,=\,\mathcal F_{1,x}\,\supset\,\cdots\, \supset\, \mathcal F_{n,x}\,\supset\,\mathcal F_{n+1,x}\,=\,0
\end{equation}
on $x\times T$ and let $\Phi\,:\,T\,\longrightarrow\, U_{\alpha}(L)$ (see \eqref{d-1}) be the morphism induced by it. Define
\begin{equation}\label{c-2}
\mathcal Q_j\,:=\,\frac{\mathcal F_{j,x}}{\mathcal F_{j+1,x}}
\end{equation}
(see \eqref{c-1}). Fix an integer $k\,\in\, \mathbb Z$ such that
$\beta_i\,=\,k\alpha_i$ is an integer for every $i$. Let $d_i\,:=\,\beta_{i+1}-\beta_i$, $1\,\leq\, i\,<\,n$
and $d_n\,:=\,k-\beta_r$. Define the line bundle
\begin{equation}\label{c-3}
{\mathcal L}_T\,:=\,(\det R\pi_{2*}\mathcal V)^k\otimes \det(\mathcal V_x)^{\ell}\otimes \bigotimes_{j=1}^n\mathcal Q_j^{d_j}
\end{equation}
(see \eqref{c-2}), where $\pi_2$ is the projection in \eqref{z1} and
$\det R\pi_{2*}\mathcal V$ is the determinant line bundle defined as 
$$\{\det R\pi_{2*}{\mathcal V}\}_t\,:=\,\{\det H^0(X,\mathcal V_t)\}^{-1}\otimes
\{\det H^1(X,\mathcal V_t)\}$$
for $t\, \in\, T$; the integer $\ell\,\in \,\Z_{>0}$ satisfies the equation
$$\sum_{i=1}^n d_i ({\rm rk}(\mathcal F_{1,x})-{\rm rk}(\mathcal F_{i,x}))+n\ell\,=\,k(d+n(1-g))$$
\cite[p.~7, Eqn.~(*)]{XS}. Since we are considering full flags, the above equation reduces to 
$$\sum_{i=1}^n d_i (n-i)+n\ell\,=\,k(d+n(1-g)).$$ It follows from \cite[Theorem 1.2]{XS} that there exists an ample line
bundle $\widehat{\mathcal L}$ over $U_{\alpha}(L)$ in \eqref{d-1} (which is unique up to algebraic equivalence) such that
$\Phi^*\widehat{\mathcal L}\,=\,{\mathcal L}_T$.
	
\subsection{Degree of a torsion-free sheaf on $U_{\alpha}(L)$}
	
Recall that, for a smooth projective variety $Y$ of dimension $m$ and an ample line bundle $N$ over it, the degree of a
torsion-free sheaf $\mathcal F$ with respect to $N$ is defined as
$$\deg \mathcal F\,:=\,c_1(\mathcal F)\cdot c_1(N)^{m-1}[Y].$$
For $Y\,=\,U_{\alpha}(L)$ (see \eqref{d-1}), we take $N\,=\,\widehat{\mathcal L}$; for $Y\,=\,\mathbb P^n$, we take $N\,=\,\mathcal O_{\mathbb P^n}(1)$. With this
notion of degree, we discuss the slope-stability of the Picard sheaf on the moduli space. 
	
\subsection{Hecke transformations and $(\ell,m)$-stability of parabolic bundles}
	
A parabolic vector bundle $E_*$ is called $(\ell,\,m)$-stable if for every proper subbundle $F$ of $E_*$, the inequality
$$\frac{{\rm par }\deg F_*+\ell}{{\rm rk\ }F}\,<\,\frac{{\rm par }\deg E_*+\ell-m}{{\rm rk\ }E}$$ holds.
It is straightforward to check that for $\ell,\,m\,\geq\, 0$, an $(\ell,\,m)$-stable parabolic bundle is $(0,\,0)$ stable; i.e., stable (in the usual sense). 
	
Let $y\,\in\, X$ be a closed point with $y\,\notin\, D$. Let $F_*$ denote the kernel of the following elementary transformation of $E_*\,\in\, U_{\alpha}(L)$
(see \eqref{d-1}) at $y$:
\begin{equation}\label{ET}
0\,\longrightarrow\, F_* \,\longrightarrow\, E_*\,\xrightarrow{\,\,\,\sigma\,\,\,}\, \mathbb C_y\,\longrightarrow\, 0.
\end{equation}
Here $F_*$ has the parabolic structure induced through the inclusion map $F\,\hookrightarrow\, E$.
The parabolic bundles satisfying the above exact sequence has the following properties: 
\begin{enumerate}
\item If $E_*$ is $(\ell,\,m)$-stable, then $F_*$ is $(\ell,\,m-1)$-stable. 

\item If $F_*$ is $(\ell,\, m)$-stable, then $E_*$ is $(\ell-1,\, m)$-stable. 
\end{enumerate}
	Therefore, if $E_*$ is $(0,1)$-stable, then $F_*$ is stable. Moreover, the condition $\det E\,=\,L$ implies
that $\det F\,=\,L(-y)$. It follows that,
	$${\rm par} \deg F_*\,=\,{\rm par}\deg E_* -1.$$
	
Now, for each homomorphism $\sigma\,:\, E\,\longrightarrow\, \mathbb C_y$ (recall that $\mathbb C_y$ denotes the torsion sheaf of
length 1 supported at $y$), we have $\ker \sigma$ to be a locally free sheaf and it has an associated parabolic
structure. Let $(E_{\sigma})_*$ denote this parabolic vector bundle of rank $n$, degree $d-1$, determinant $L(-y)$ whose
parabolic structure is induced by $E$ via the inclusion map. Note that for $\lambda\in\mathbb C$, $E_{\sigma}\,=\,E_{\lambda\sigma}$. Therefore, the set of
all homomorphisms $\sigma$ that gives distinct parabolic bundles is the projective space $\mathbb P (E_y^*)$ (the space of
lines in $E_y^*$). In other words, fixing $E$ and $y$, we obtain a family of parabolic vector bundles $\mathcal F$ on $X$ parametrized
by $\mathbb P(E_y^*)$, which fits into the exact sequence
\begin{equation}\label{family}
0\,\longrightarrow\,\mathcal F\,\longrightarrow\, \pi_1^*(E_*)\,\longrightarrow\, \mathcal{O}_{\{y\}\times \mathbb P(E_y^*)}(1)
\,\longrightarrow\, 0.
\end{equation}
The vector bundle $\mathcal F$ comes with a natural parabolic structure over the smooth divisor $D\times \mathbb P(E_y^*)$.
	
For our purposes, we take $E_*\,\in\, U_{\alpha}(L(y))$. From the above discussions we know that if $E_*$ is a $(0,\,1)$-stable bundle,
then $F_*$ is stable. In other words, if $E_*$ is a $(0,1)$-stable bundle, then $\mathcal F$ is a family of stable parabolic
vector bundles of rank $n$, degree $d$ and determinant $L$ which is parametrized by $\mathbb P(E_y^*)$. By the universal property
of the moduli space $U_{\alpha}(L)$ (see \eqref{d-1}), there is a morphism 
\begin{equation}\label{psi}
    \Psi_{E_*,y}\,\,:\,\, \mathbb P(E_y^*)\,\,\longrightarrow\,\, U_{\alpha}(L)
\end{equation}
such that 
\begin{equation}\label{FandU}
\mathcal F\,\,\cong\,\,({\rm Id}_X\times\Psi_{E_*,y})^*{\mathcal U}_L\otimes\pi_2^*\mathcal O_{\mathbb P(E_y^*)}(-j)
\end{equation}
(see \eqref{d-2}) for some integer $j$, and they are equivalent as families. 
	
	The following result ensures that we can always construct such families. 
	
\begin{lemma}\label{open}
The $(0,\,1)$-stable bundles and $(1,\,0)$-stable bundles form a nonempty Zariski open subset of $U_{\alpha}^s$ (see \eqref{c4}).
\end{lemma}

\begin{proof}
See \cite[Proposition 2.7]{UBIB2}. Note the Assumption \ref{ams1}(2) ensures that
Proposition 2.7 of \cite{UBIB2} applies in our context.
\end{proof}
	
\begin{lemma}\label{inj}
The map $\Psi_{E_*,y}$ in \eqref{psi} is injective. 
\end{lemma}
	
\begin{proof}
It suffices to show that $h^0(Hom(F_*,\,E_*))\,:=\, \dim H^0(Hom(F_*,\,E_*))\,=\,1$. Firstly, note that any homomorphism $F\,\xrightarrow{\,\,\,\sigma\,\,\,}
\,E$ is an isomorphism away from $y$ and is of maximal rank $n$. For otherwise, we can factorize $\sigma$ as
\begin{center}
\begin{tikzcd}
F_*\ar[r] & G'\ar[r]\ar[d] & 0\\
E_* & G\ar[l] & 0\ar[l]
\end{tikzcd}
\end{center} 
where $G'\,=\,F/\ker \sigma$ and $G\,=\,Im(\sigma)$. Then ${\rm par}\deg G_*
\,\geq \,{\rm par}\deg F_*$ and ${\rm rk}\ G_*\,=\,{\rm rk}\ F_*$ and we get the following inequality: 
$$\frac{{\rm par}\deg G_*}{{\rm rk}\ G_*} \,\geq\, \frac{{\rm par}\deg
G'_*}{{\rm rk}\ G'_*}\,>\, \frac{{\rm par}\deg F_*}{{\rm rk}\ F_*}\,=\, \frac{{\rm par}\deg E_*-1}{{\rm rk}\ E_*},$$
which contradicts the $(0,1)$-stability condition of $E_*$.

Fix $x\,\neq\, y$. Suppose two linearly independent homomorphisms $\sigma$ and $\tau$ exist. Since they are of maximal rank, we can find $a,\,
b\,\in \,\mathbb C$ such that $\mu\,:=\,a\sigma+b\tau$ induces a singular element in $Hom(F_x,\,E_x)$. But the only zero of the homomorphism $$\wedge^n\mu\, :\, \det F\,\cong\, \det E \otimes \mathcal O(-y)
\,\longrightarrow\, \det E$$ is $y$, which is a contradiction. This completes the proof.
\end{proof}

\begin{remark}\label{positive}
As a consequence of Lemma \ref{inj}, we have $\Psi_{E_*,y}^*
\widehat{\mathcal L}\,=\,\mathcal O_{\mathbb P(E_y^*)}(\beta)$ for
some $\beta\,>\,0$.
\end{remark}

\subsection{Main theorem}

\begin{theorem}\label{main}
Let $X$ be a non-singular complex projective curve of genus $g\,\geq\,2$. Fix integers $r,\,n,\,d$
such that $r\,\geq\, 1$ and $n\,\geq\, 2$. Also, fix $r$ distinct points $x_1,\,\cdots,\, x_r$ on $X$ and a line bundle $L$ on $X$ of degree $d$.
Assume that the two conditions in Assumption \ref{ams1} hold.
Let $U_{\alpha}(L)$ be as in \eqref{d-1} and $\mathcal{W}_L$ denote the parabolic Picard sheaf (see \eqref{k1}). Then $\mathcal{W}_L$ is stable with
respect to the polarization $\widehat{\mathcal L}$.
\end{theorem}
 
We first make some observations and prove a lemma required for the proof of Theorem \ref{main}.
 
\noindent For any $y\,\notin\, D$, consider the projective bundle 
\begin{equation*}
    P_y\,:=\,{\mathbb P}(\mathcal U_L|_{\{y\}\times U_{\alpha}(L)})
\end{equation*}
(see \eqref{d-2}) over $\{y\}\times U_{\alpha}(L)$.
The elements of $P_y$ are of the form $(F_*,\,l)$ with $F_*\,\in\, U_{\alpha}(L)$ and $l\,\in\, \mathbb P(F_y)$. An
element $(F_*,\,l)$ of $P_y$ also corresponds to a short exact sequence \eqref{ET}.
Let
$$H_y\,:=\,\{(F_*,\,l)~\vert \ {\rm the\ bundle\ } E_*\ \mbox{defined in\ \eqref{ET} is\ (0,1)-stable}\}.$$
Then there are morphisms 
\begin{equation}\label{pq}
p\,:\,H_y\,\longrightarrow\, U_{\alpha}(L) {\rm ~and~} q\,:\,H_y\,\longrightarrow\, U_{\alpha}(L(y))
\end{equation}
defined by $(F_*,\,l)\,\longmapsto\, F_*$ and $(F_*,\,l)\,\longmapsto\, E_*$, respectively. Let $W$ denote the open subset (by Lemma \ref{open})
\begin{equation}\label{W}
   W =\{E_*\in U_{\alpha}(L(y))\vert\ E_*\ \mbox{is\ (0,1)-stable}\}.
\end{equation}
Then the image of $q$ is $W$ and the
fiber of $q$ over $E_*\,\in\, W$ can be identified with ${\mathbb P}(E_y^*)$. Note that the restriction of $p$ to $\mathbb P(E_y^*)$
is precisely the map $\Psi_{E_*,y}$. By Lemma \ref{inj}, $\Psi_{E_*,y}$ maps isomorphically onto its image, say 
\begin{equation}\label{PEy}
    P(E_*,y)\,:=\,
\Psi_{E_*,y}(\mathbb P(E_y^*)).
\end{equation}

\begin{lemma}\label{generic}
Let $\mathcal G$ be a subsheaf of $\mathcal W_L$ with $0\,<\,{\rm rk\ }\mathcal G\,<\,{\rm rk\ }\mathcal W_L$.
For $k\,\in\,\mathbb N$, let $y_1,\,\cdots,\,y_k\,\in\, X\setminus D$. There is a non-empty open subset $V$ of $U_{\alpha}(L)$ such that
for all $F_*\,\in\, V$ the following hold:
\begin{enumerate}
\item $\mathcal G$ is locally free at $F_*$;

\item the homomorphism of fibers $\mathcal G_{F_*}\,\longrightarrow\, (\mathcal W_L)_{F_*}$ is injective;

\item for all $y_i$ and a generic line $l$ in $F_{y_i}$, the parabolic bundle $E_*$ associated to $(F_*,\,l)$ is $(0,1)$-stable
and $\mathcal G$ is locally free at every point of $P(E_*,\,y_i)$ (see \eqref{PEy}) outside a subvariety of codimension at least 2.
\end{enumerate}
\end{lemma}

\begin{proof}
The first two conditions are in fact open conditions and hence define an open subset of the moduli space.

Proof of (3): Let $S$ denote the singular set of $\mathcal G$. We show that for a fixed $y\,\notin\, D$ and general $E_*\,\in\, U_{\alpha}(L(y))$, either
$P(E_*, \,y)$ is empty or $\dim P(E_*,\,y)\cap S\,\leq\, n-3$. It suffices to consider the case when $S$ is irreducible. Let $S'
\,=\,p^{-1}(S)$ be the inverse image of $S$ in $H_y$. The fiber of $q\vert_{S'}$ over $E_*$ (see \eqref{pq}) can be identified with $P(E_*,y)\cap S$. Now $S'$ is
either empty or an irreducible variety of dimension $\dim S+n-1$ where $\dim U_{\alpha}(L)-\dim S\,\geq 2$. If $q(S')$ is not dense in $W$ (see \eqref{W}), then
for general $E_*$, we have $P(E_*,y)\cap S\,=\,\emptyset$. If $q(S')$ is dense in $W$, then the general fiber of $q\vert_{S'}$ has dimension, say $\delta$,
\begin{align*}
\delta &= \dim S'-\dim W\\
&=\, \dim S + n-1 -\dim U_{\alpha}(L(y))\\
&=\, \dim S+ n-1-\dim U_{\alpha}(L)\\
& \leq\, n-3.
\end{align*}
This completes the proof.
\end{proof}
	
\begin{proof}[{Proof of Theorem \ref{main}}]
Let $\mathcal G$ be a (saturated) subsheaf of $\mathcal W_L$ such that $0\,<\,{\rm rk}(\mathcal G)
\,<\,{\rm rk}(\mathcal W_L)$. For $k\,>\,\mu(F_*)$, fix points $y_1,\,\cdots,\,y_k\,\in\, X\setminus D$.
Choose $F_*\,\in\, U_{\alpha}(L)$ such that $F_*$ satisfies the conditions of Lemma \ref{generic}. Then by Lemma \ref{generic}(2), for a fixed
$0\,\neq\, v\,\in\, \mathcal G_{F_*}$, its image $s$ in $(\mathcal W_L)_{F_*}\,=\,H^0(X,\,F_*)$ is non-zero. Observe that if $s(y_i)\,=\,0$ for all
$i$, then $s\,\in\, H^0(X,\, F_*(-y_1-\dots -y_k))$ and hence it induces a homomorphism $\mathcal O_X\,\longrightarrow\, F_*(-y_1-\dots -y_k)$. Since
$F_*$ is stable, we have $F_*(-y_1-\dots -y_k)$ is stable and hence $H^0(X,\, F_*(-y_1-\dots -y_k))\,=\,0$, a contradiction to our choice of $s$.
Therefore there is an $i$ such that $s(y_i)\,\neq\, 0$; let $y\,:=\,y_i$.

Now for the chosen $F_*$ and $y$, using Lemma \ref{generic}(3), we can also choose a line $l\,\subset\, F_y$ such that $s(y)
\,\notin\, l$ and $\mathcal G$ is 
locally free on $P(E_*,\, y)$ except for some subvariety of codimension 2, where $E_*$ is the
vector bundle associated to $(F_*,\,l)$. We can reconstruct $(F_*,\,l)$ 
from $E_*$ using the following diagram:
\begin{center}
\begin{eqnarray}\label{D1}
\begin{tikzcd}
& 0\arrow{d} & 0\arrow{d} & & \\
& E_*(-y)\arrow{d}\arrow[equal]{r}& E_*(-y)\arrow{d}& & \\
0\arrow{r}& F_*\arrow{r}\arrow{d} & E_*\arrow{r}{\sigma}\arrow{d} & \mathbb C_y\arrow[equal]{d}\arrow{r} & 0\\
0\arrow{r} & F_y/l\arrow{r}\arrow{d}& E _y\arrow{r}\arrow{d} & \mathbb C_y\arrow{r}& 0\\
& 0 & 0 & & 
\end{tikzcd}
\end{eqnarray}
\end{center}
Varying $\sigma\,\in\,\mathbb P(E_y^*)$, we obtain a family of parabolic bundles $\mathcal F$ as in \eqref{family} that fits into the
following diagram:
\begin{center}
\begin{eqnarray}\label{D2}
\begin{tikzcd}
& 0\arrow{d} & 0\arrow{d} & & \\
& \pi_1^*E_*(-y)\arrow{d}\arrow[equal]{r}& \pi_1^*E_*(-y)\arrow{d}& & \\
0\arrow{r}& \mathcal F\arrow{r}\arrow{d} & \pi_1^*E_*\arrow{r}\arrow{d} & \mathcal{O}_{\{y\}\times \mathbb P(E_y^*)}(1)\arrow[equal]{d}\arrow{r} & 0\\
0\arrow{r} & \Omega_{\{y\}\times \mathbb P(E_y^*)}(1)\arrow{r}\arrow{d}& E _y\otimes \mathcal{O}_{\{y\}\times \mathbb P(E_y^*)}\arrow{r}\arrow{d} & \mathcal{O}_{\{y\}\times \mathbb P(E_y^*)}(1)\arrow{r}& 0\\
& 0 & 0 & & 
\end{tikzcd}
\end{eqnarray}
\end{center}
Taking the direct image of the above diagram by $\pi_2$ and using \eqref{FandU}, we obtain the following diagram on $\mathbb P(E_y^*)$:
\begin{center}
\begin{eqnarray}
\begin{tikzcd}
& 0\arrow{d} & 0\arrow{d} & & \\
& H^0(E_*(-y))\otimes\mathcal O_{\mathbb P(E_y^*)}\arrow{d}\arrow[equal]{r}& H^0(E_*(-y))\otimes\mathcal O_{\mathbb P(E_y^*)}\arrow{d}& & \\
0\arrow{r}& \Psi_{E_*,y}^*\mathcal W_L(-j)\arrow{r}\arrow{d} & H^0(E_*)\otimes \mathcal O_{\mathbb P(E_y^*)}\arrow{r}\arrow{d} & \mathcal{O}_{\{y\}\times \mathbb P(E_y^*)}(1)\arrow[equal]{d}\arrow{r} & 0\\
0\arrow{r} & \Omega_{\mathbb P(E_y^*)}(1)\arrow{r}\arrow{d}& E _y\otimes \mathcal{O}_{\mathbb P(E_y^*)}\arrow{r}\arrow{d} & \mathcal{O}_{\{y\}\times \mathbb P(E_y^*)}(1)\arrow{r}& 0\\
& H^1(E_*(-y))\otimes\mathcal O_{\mathbb P(E_y^*)}\arrow[equal]{r}\arrow{d} & H^1(E_*(-y))\otimes\mathcal O_{\mathbb P(E_y^*)}\arrow{d} & & \\
& 0 & 0 & & 
\end{tikzcd}
\end{eqnarray}
\end{center}
It is clear from the above diagram that $\Psi_{E_*,y}^*\mathcal W_L(-j)$ has degree -1. 
Let $\mathcal G'$ denote the image of $\Psi_{E_*,y}^*\mathcal G(-j)\,\longrightarrow\, \Psi_{E_*,y}^*\mathcal W_L(-j)$. By the choice of $F_*$,
we have $\mathcal G'$ to be isomorphic to $\Psi_{E_*,y}^*\mathcal G(-j)$ away from a subvariety of codimension 2 and hence
$$\deg \Psi_{E_*,y}^*\mathcal G(-j)\,\,=\,\, \deg \mathcal G'.$$

Let $K$ and $I$ respectively denote the kernel and the image of the induced homomorphism $\mathcal G'\,\longrightarrow\,
\Omega_{\mathbb P(E_y^*)}(1)$. Since $l$ is the image of $E_*(-y)_y$ in $(F_*)_y$ (see Diagram \eqref{D1}) and $s(y)\notin l$, we get
that $s\,\notin\, H^0(E_*(-y))$. It follows that the image of $s$ in $I$ is non-zero and hence we have $I\,\neq\, 0$. We can conclude that 
$\deg I\,\leq\, -1$ as $\Omega_{\mathbb P(E_y^*)}(1)$ is stable and has degree $-1$. We can also conclude that $\deg K \,\leq\, 0$
because $K$ is a subsheaf of the trivial sheaf $H^0(E_*(-y))\otimes\mathcal O_{\mathbb P(E_y^*)}$. Thus, we obtain that
$\deg \mathcal G'\,\leq\, -1$. 

As a consequence, we have
$${\frac{\deg\Psi_{E_*,y}^*\mathcal G(-j)}{{\rm rk}\ \mathcal G}}\,=\,\frac{-1}{{\rm rk}\ \mathcal G}
\,<\,\frac{-1}{{\rm rk}\ \mathcal W_L}\,=\,\frac{\deg \Psi_{E_*,y}^*\mathcal W_L(-j)}{{\rm rk}\ \mathcal W_L}.$$
Now using Remark \ref{positive}, we can conclude that
$$\frac{\deg \mathcal G}{{\rm rk}\ \mathcal G}\,=\,{\frac{\deg\Psi_{E_*,y}^*\mathcal G}{{\rm rk}\ \mathcal G}}
\,<\,\frac{\deg \Psi_{E_*,y}^*\mathcal W_L}{{\rm rk}\ \mathcal W_L}\,=\,\frac{\deg \mathcal W_L}{{\rm rk}\ \mathcal W_L}.$$
This completes the proof.
\end{proof}

\end{document}